\documentclass[11pt]{amsart}

\usepackage[all]{xypic}
\usepackage{tikz}
\usepackage{ragged2e}
\usetikzlibrary{arrows} 
\usetikzlibrary{decorations.markings}
\usepackage{graphicx}
\usepackage{bm}
\usepackage{float}
\usepackage{epsf}
\usepackage{verbatim} 
\usepackage{amsmath}
\usepackage{amsfonts}
\usepackage{bbm}
\usepackage{amssymb}
\usepackage{mathrsfs}
\usepackage{amsthm}
\usepackage{newlfont}
\usepackage{enumitem}
\usepackage[new]{old-arrows}
\usepackage{booktabs}
\usepackage{enumitem}
\usepackage{makecell}
\usepackage[multiple]{footmisc}
\usepackage[answerdelayed,lastexercise]{exercise}
\usepackage[hidelinks]{hyperref}
\usepackage{mleftright}

\usepackage{fnpct}
\usepackage{colonequals} % From Poonen's "writing"
\usepackage{stmaryrd}

\usepackage{apptools}
\usepackage{chngcntr}
\newtheorem{thm}{Theorem}[section]
\newtheorem{prop}[thm]{Proposition}
\newtheorem{lem}[thm]{Lemma}
\newtheorem{cor}[thm]{Corollary}

\theoremstyle{definition}
\newtheorem{defn}[thm]{Definition}
\newtheorem{notn}[thm]{Notation}
\newtheorem{example}[thm]{Example}

\theoremstyle{remark}
\newtheorem{rem}[thm]{Remark}

\numberwithin{equation}{section}
\numberwithin{thm}{section}

\makeatletter 
\newcommand\mynobreakpar{\vspace{0.02in}\par\nobreak\@afterheading}  
\makeatother

\makeatletter
\@addtoreset{footnote}{section}
\makeatother

\usepackage{xcolor}
\hypersetup{
	colorlinks,
	linkcolor={red!50!black},
	citecolor={blue!50!black},
	urlcolor={blue!80!black}
}

\AtBeginDocument{%
	\def\MR#1{}
}

\DeclareMathOperator{\ICM}{ICM}

\DeclareMathOperator{\re}{Re}
\DeclareMathOperator{\im}{Im}
\DeclareMathOperator{\Mat}{Mat}

\DeclareMathOperator{\GL}{GL}

\DeclareMathOperator{\Fix}{Fix}

\DeclareMathOperator{\Cl}{Cl}

\DeclareMathOperator{\Tr}{Tr}

\DeclareMathOperator{\diag}{diag}

\DeclareMathOperator{\disc}{disc}

\renewcommand{\Fix}{\mathrm{Fix}}

\renewcommand{\diag}{\mathrm{diag}}

\newcommand{\fp}{\mathfrak{p}}

\newcommand{\cC}{\mathcal{C}}

\newcommand{\cI}{\mathcal{I}}

\newcommand{\cM}{\mathcal{M}}

\newcommand{\cO}{\mathcal{O}}
\newcommand{\cP}{\mathcal{P}}

\newcommand{\C}{\mathbb{C}}

\newcommand{\F}{\mathbb{F}}

\newcommand{\Q}{\mathbb{Q}}

\newcommand{\Z}{\mathbb{Z}}

\newcommand{\oOm}{\overline{\Omega}}

\newcommand{\eps}{\varepsilon}

\newcommand{\To}{\longrightarrow}
\newcommand{\bs}{\setminus}

\newcommand{\Mod}[1]{\ (\mathrm{mod}\ #1)}
\newcommand{\abs}[1]{\lvert#1\rvert}

\newcommand{\ind}{\mathbbm{1}}

\title{On ideal class groups of totally degenerate number rings}

\author{Ruben Hambardzumyan}
\address{Faculty of Mathematics and Mechanics, Yerevan State University, Yerevan, Armenia}
\email{ruben.hambardzumyan2@edu.ysu.am}

\author{Mihran Papikian}
\address{Department of Mathematics, Pennsylvania State University, University Park, Pennsylvania, United States of America}
\email{papikian@psu.edu}

\thanks{The first author was supported by the Science Committee of Republic of Armenia  (Research project No 23RL-1A027).} 
\thanks{The second author was supported in part by the Simons Foundation, award number MPS-TSM-00008093.} 


\begin{document}
	
	\maketitle
	
	\begin{abstract}
		Let $\chi(x)\in \Z[x]$ be a monic polynomial whose roots are distinct integers. We study the ideal class monoid and the ideal class group of the ring 
		$\Z[x]/(\chi(x))$. We obtain formulas for the orders of these objects,  
		and study their asymptotic behavior as the discriminant of $\chi(x)$ tends to infinity, in analogy with the Brauer-Siegel theorem. Finally, we describe the structure of the ideal class group when the degree of $\chi(x)$ is $2$ or $3$. 
	\end{abstract}
	
% ----------------------------------------------------------------------

\section{Introduction} The study of class groups of the rings of integers of number fields can be traced back to Gauss. The class group $\Cl(\cO_K)$ of the ring of integers $\cO_K$ of a number field $K$ is finite, 
and, from the perspective of abstract algebra, it measures the failure of the 
uniqueness of factorization in $\cO_K$. On the other hand, in  number theory, $\Cl(\cO_K)$ is a fundamental 
invariant of $K$ related to many famous questions such as  the existence of unramified abelian extensions of $K$, the residue at $s=1$ of the Dedekind 
zeta function $\zeta_K(s)$ of $K$, Fermat's Last Theorem, and many others. 

The class group can be defined also for an order in $\cO_K$, i.e., for a subring $R\subseteq \cO_K$ with the same identity 
such that $\cO_K/R$ is finite. One new subtlety here is that not all fractional ideals of $R$ are invertible. The class groups of orders 
naturally arise in class field theory and can be related to the class group of the maximal order $\cO_K$; see  \cite{Neukirch}, \cite{Cox},  \cite{KL}. %; see \cite[p. 81]{Neukirch}. 
Generalizing this further, one can study the class groups of orders in products of numbers fields. An analytic class number formula 
in this generality was proved recently in \cite{JP}, and there have been recent advances in computational aspects of the theory of ideals of such orders; 
see \cite{EHO}, \cite{Marseglia}, \cite{BHJ}.

In this paper we study the class groups of orders in the direct product of $n$ copies of $\Q$, that is, the ``totally degenerate" case. 
The phrase ``totally degenerate number rings" in the title refers specifically to such orders, and should not be confused with the 
more conventional use of ``number ring" to denote the ring of integers of a number field (as, for example, in \cite{Marcus}).

Let $a_1<a_2<\cdots<a_n$ be distinct integers and put  
$$
\chi(x)=(x-a_1)(x-a_2)\cdots(x-a_n). 
$$
Let $R=\Z[x]/(\chi(x))$ be the quotient of the polynomial ring $\Z[x]$ by the ideal generated by $\chi(x)$. Then $R$ 
is a $\Z$-order in $$K=\Q\otimes_{\Z} R\cong \Q[x]/(\chi(x))\cong \prod_{i=1}^n\Q.$$
The normalization of $R$ in $K$ is $\prod_{i=1}^n\Z$.

The fractional ideals in $R$, up to linear equivalence, form a finite monoid under multiplication of ideals, the \textit{ideal class monoid}, which 
we denote by $\ICM(R)$ (following the notation in \cite{Marseglia}). The \textit{ideal class group} $\Cl(R)$ is the group of 
units of $\ICM(R)$. 

To state the main results of this paper, we need to introduce some notation.  
Let $k$ be an integer such that $1\leq k\leq n-1$.  
For a subset $S\subset \{1, 2, \dots, n\}$ of order $k+1$, let 
	$$
	\Delta_S = \prod_{\substack{i, j\in S\\ i<j}} (a_j-a_i). 
	$$
	Let 
	\begin{equation}\label{eqDeltak}
	\Delta_k = \gcd_{ \abs{S}=k+1} (\Delta_S),
	\end{equation}
	where $S$ runs over all subsets of $\{1, \dots, n\}$ of order $k + 1$. Note that 
	$$\Delta\colonequals \Delta_{n-1}=\prod_{i<j}(a_j-a_i)$$ 
	is the square-root of the discriminant of $\chi(x)$ 
	and 
	$$
	\Delta_1 = \gcd_{i<j} (a_j-a_i). 
	$$
Let $\varphi$ be Euler's totient function. 

Using the class number formula in \cite{JP}, we prove the following theorem. 

\begin{thm}\label{thmMain1}
	If $a_n-a_1\geq 4$, then 
	$$
	\abs{\Cl(R)} = \frac{\varphi(\Delta)}{2^{n-1}} \prod_{l=1}^{n-2} \frac{\varphi(\Delta_l)}{\Delta_l}.
	$$
	If $a_n-a_1<4$, then $\abs{\Cl(R)} = 1$. 
\end{thm}

Next, using the above formula, we derive the asymptotic behavior of $\abs{\Cl(R)}$ as $\Delta$ tends 
to infinity. This can be interpreted as an analogue of the celebrated Brauer-Siegel theorem in this context since the regulator of $R$ is trivial. 

\begin{thm}\label{thmMain2}
	\begin{align*}
		& \liminf_{\Delta\to \infty} \frac{\abs{\Cl(R)} \cdot 2^{n-1}}{\varphi(\Delta)} = 
		\begin{cases}
			1, & \text{if } n=2; \\ 
			0, & \text{if } n\geq 3.
		\end{cases} \\ 
		& \limsup_{\Delta\to \infty} \frac{\abs{\Cl(R)} \cdot 2^{n-1}}{\varphi(\Delta)} = \prod_{\substack{p \le n - 2 \\ p \text{ is prime}}} (1 - p^{-1})^{n - 1 - p}.
	\end{align*}
\end{thm}

We also prove an asymptotic formula for the order of $\ICM(R)$, although a closed formula for $\abs{\ICM(R)}$ seems much harder to obtain. 

\begin{thm}\label{thmMain3}
	For any $n\geq 1$, 
	$$
	\lim_{\Delta\to \infty}\frac{\abs{\ICM(R)}\cdot 2^{n-1}}{\Delta}=1. 
	$$	
\end{thm}

To prove Theorem \ref{thmMain3} we use the well-known bijection between $\ICM(R)$ and the conjugacy classes of matrices in $\Mat_n(\Z)$ 
with characteristic polynomial $\chi(x)$; see \cite{LM}.   
Using this bijection, we also compute the structure of the group $\Cl(R)$ for $n=2,3$. We state a somewhat simplified version of our result (the actual theorems 
in Section \ref{sExamples} make no assumptions except that $n=2, 3$):  

\begin{thm}\hfill
		\begin{itemize} 
			\item[(i)] Assume $n=2$ and $(a_2-a_1)\geq 3$. The group $\Cl(R)$ is isomorphic to the quotient of the group $(\Z/(a_2-a_1))^\times$ by the subgroup $\{\pm 1\}$.   
			\item[(ii)] Assume $n=3$, $a_3-a_1\geq 4$, and $\Delta_1 = 1$. The group 
			$\Cl(R)$ is isomorphic to the quotient of the group
			\[
			\left ( \Z/(a_2 - a_1) \right)^\times \times \left ( \Z/(a_3 - a_2) \right )^\times \times \left ( \Z/(a_3 - a_1) \right )^\times
			\]
			by the subgroup $\left\{ (1, 1, 1), (1, -1, -1), (-1, 1, -1), (-1, -1, 1) \right \}.$
			\end{itemize}
\end{thm}

The paper is organized as follows. In Section \ref{sPrelim}, for a separable monic polynomial $\chi(x)$ of degree $n$,
we recall the definition of a fractional ideal of $R=\Z[x]/(\chi(x))$ and the 
 ideal class monoid $\ICM(R)$. We also recall the bijection between $\ICM(R)$ 
 and the conjugacy classes of matrices in $\Mat_n(\Z)$ with characteristic polynomial $\chi(x)$ due to Latimer and MacDuffee. 
 In Section \ref{sICM}, assuming all the roots of $\chi(x)$ are in $\Z$, we study $\ICM(R)$ using the conjugacy classes of matrices in  
 $\Mat_n(\Z)$; here we prove Theorem \ref{thmMain3}. In Section \ref{sICG}, we recall the Class Number Formula of Jordan and Poonen, 
 and prove Theorems \ref{thmMain1} and \ref{thmMain2}. In Section \ref{sExamples}, we describe $\Cl(R)$ as an abelian group  
 for $n=2,3$, using a combination of computations with matrices and ideals.

\subsection*{Acknowledgments} 
The authors are grateful to the anonymous referee for a careful reading of the manuscript and for many valuable suggestions. 
This work was carried out as part of a project within the Honors Program in Mathematics of the Armenian Society of Fellows (ASOF). 
We thank ASOF for organizing the program and for its financial support.

% ----------------------------------------------------------------------

\section{Preliminaries}\label{sPrelim}  

For a discussion of $\Z$-orders in number fields and their class groups, the reader might consult \cite{Neukirch} and \cite{Cohen}.  
A discussion of these concepts in products of number fields, as well as computational aspects of the theory, can be found in \cite{Marseglia}. 

Let $\chi(x)\in \Z[x]$ be a monic separable polynomial of degree $n$, i.e., a polynomial without repeated roots in $\C$. 
Let $R=\Z[x]/(\chi(x))$ be the quotient of the polynomial ring $\Z[x]$ by the ideal generated by $\chi(x)$, and let 
$K=\Q\otimes_\Z R \cong \Q[x]/(\chi(x))$. 
Then $R$ is naturally a subring of $K$.  % (see Exercise 26 on page 377 in \cite{DF}). 
Because $\chi(x)$ is separable, it decomposes as a product $f_1(x)\cdots f_s(x)$ of monic polynomials in $\Z[x]$, where  $f_i\neq f_j$ for $i\neq j$ and 
each $f_i(x)$ is irreducible in $\Q[x]$. The Chinese Remainder Theorem %(cf. \cite[p. 265]{DF}) 
implies that 
$$
K\cong \prod_{i=1}^s \Q[x]/(f_i(x)) \cong \prod_{i=1}^s K_i
$$ 
is a product of finite field extensions of $\Q$, so $K$ is an \'etale $\Q$-algebra of dimension $n$. 
On the other hand, generally $R$ does not have a similar splitting, i.e., $R$ is not isomorphic to the direct product 
$R_1\times \cdots \times R_s$ of $\Z$-orders $R_i\subset K_i$, $i=1, \dots, s$.

	A \textit{fractional $R$-ideal} in $K$ is an $R$-submodule $I\subset K$ which has rank $n$ as a $\Z$-module. Given 
	two fractional $R$-ideals $I$ and $J$, we can form their product 
	$$I\cdot J=\{i_1j_1+\cdots+i_mj_m\mid m\geq 1, i_1, \dots, i_m\in I, j_1, \dots, j_m\in J\}.$$ This is again a fractional $R$-ideal. 
    Denote the monoid of fractional $R$-ideals by $\cI(R)$. Two fractional $R$-ideals $I$ and $J$ are (linearly) \textit{equivalent}, denoted $I\sim J$, if $I=kJ$ for some $k\in K^\times$, where $K^\times$ denotes the group of units of $K$.  
    The \textit{ideal class monoid} of $R$, denoted $\ICM(R)$, is 
    the monoid of equivalence classes of fractional $R$-ideals. 
    A fractional ideal of the form $I=\alpha R$, $\alpha\in K^\times$, is called \textit{principal}.  
The group $\cP(R)$ of principal fractional $R$-ideals acts by multiplication on $\cI(R)$ and 
    $$\ICM(R)=\cI(R)/\cP(R).$$ 
    
    Given two fractional $R$-ideals $I$ and $J$, their \textit{ideal quotient} is   
    $$(I:J)=\{k\in K\mid k J\subseteq I\}.$$ 
	A fractional $R$-ideal $I$ is said to be \textit{invertible} if $IJ=R$ for some 
    fractional $R$-ideal $J$. Note that if such $J$ exists, then $J=(R:I)$.  
    Generally not every fractional $R$-ideal is invertible, so $\ICM(R)$ is not a group. The \textit{ideal class group} of $R$ is the subgroup $\Cl(R)$ of $\ICM(R)$ consisting of equivalence classes of invertible fractional $R$-ideals.

    \begin{rem}\label{remGorenstein}
It is easy to check that if $I$ is an invertible fractional $R$-ideal, then $(I:I)=R$; see Lemma 2.5 in \cite{Marseglia}. Moreover, since our $R$ is a Gorenstein ring, this is also a sufficient condition, i.e., $I$ is an invertible fractional $R$-ideal if and only if $(I:I)=R$; see Proposition 2.11 and Corollary 2.12 in \cite{Marseglia}. 
    \end{rem}

\begin{example}
	Let $$R=\Z[x]/(x(x-p))\cong \{(m, n)\in \Z\times \Z\mid m\equiv n\Mod{p}\},$$ where $p$ is prime. Let $I\lhd R$ be the ideal generated by $\{(p, 0), (0, p)\}$. Note that $R/I\cong \F_p$, so $I$ is maximal. Next, for an arbitrary $(m,n)\in \Z\times \Z$, we have $(m, n)I\subseteq I$, so $(I:I)=\Z\times \Z$, which is strictly larger than $R$. By Remark \ref{remGorenstein}, $I$ is not invertible. 
\end{example}

\begin{defn}
	Two $n\times n$ matrices $S_1, S_2\in \Mat_n(\Z)$ are \textit{conjugate} over $\Z$ (or \textit{$\Z$-conjugate}) if there exists $T\in \GL_n(\Z)$ such that $S_2=TS_1T^{-1}$. %Conjugacy over $\Z$ is an equivalence relation, and 
	Matrices which are conjugate over $\Z$ have the same characteristic and minimal polynomials; the converse is not true in general even when the matrices are conjugate over $\Q$. 
\end{defn}

%\begin{rem} Let $A\in \Mat_n(\Q)$ and let $\chi_A(x)=\det(xI_n-A)\in \Q[x]$ be its characteristic polynomial. 
%	Assume $\chi_A(x)$ is a separable polynomial. It is a basic fact from linear algebra 
%	that $A$ and $B$ are conjugate in $\Mat_n(\Q)$, i.e., $B=XAX^{-1}$ for some $X\in \GL_n(\Q)$, if and only if 
%	$\chi_A(x)=\chi_B(x)$. This is no longer the case if one considers conjugacy over $\Z$. 
%\end{rem}

There is a bijection between $\ICM(R)$ and the set of $\Z$-conjugacy classes of matrices with characteristic polynomial $\chi(x)$.  
The following lemma will be used in the proof of that bijection. 

\begin{lem}\label{lemModICM} 
We say that an $R$-module $M$ is of \textbf{rank 1} if $M\cong \Z^n$ as a $\Z$-module and $M\otimes_\Z \Q\cong K$ as a $K$-module. 
We have:
	\begin{enumerate}
		\item Any $R$-module of rank $1$ is isomorphic to a fractional ideal of $K$. 
		\item Two $R$-modules of rank $1$ are isomorphic if and only if their corresponding fractional ideals are equivalent. 
	\end{enumerate}
\end{lem}
\begin{proof}
	(1) Let $M$ be an $R$-module of rank $1$. By definition, $V\colonequals M\otimes_\Z \Q\cong K$. 
	Since $M$  is an $R$-submodule of $V$, 
	we can identify $M$ with an $R$-submodule of $K$. Since  $M$ has rank $n$ over $\Z$, it is a fractional ideal.  
	
	(2) Suppose $I$ and $J$ are equivalent fractional $R$-ideals of $K$, so $J=k I$ for some $k\in K^\times$. Then $I$ and $J$ are isomorphic $R$-modules of rank $1$, with 
	the isomorphism $\theta\colon I\to J$ given by $i\mapsto ki$.  
	
	Now suppose $I$ and $J$ are $R$-modules of rank $1$ and $\theta\colon I\to J$ is an isomorphism of $R$-modules. The isomorphism $\theta$ extends linearly to an isomorphism $\theta\colon I\otimes_\Z \Q\to J\otimes_\Z \Q$ of $R\otimes_\Z\Q$-modules. Since $I\otimes_\Z \Q\cong J\otimes_\Z \Q\cong R\otimes_\Z \Q=K$, we get a $K$-linear isomorphism $\theta\colon K\to K$. Such an isomorphism is given by multiplication by some $k\in K^\times$. Hence the images of $I$ and $J$ in $K$ satisfy $J=kI$. 
\end{proof}

The next theorem is originally due to Latimer and MacDuffee \cite{LM}, where it is stated and proved in a somewhat 
outdated terminology. A nice exposition of the theorem, with many detailed examples, is given in \cite{Conrad}. (In \cite{Conrad}, it is  
assumed that $\chi(x)$ is irreducible but this assumption is not necessary.) Another proof (of a more general statement) is given in \cite[$\S$8]{Marseglia}. 
We give a proof for the sake of completeness and because this construction will be used in Section \ref{sExamples}. 

\begin{thm}\label{thmLM}
	There is a bijection between $\ICM(R)$ and the set of $\Z$-conjugacy classes of matrices in $\Mat_n(\Z)$ 
	with characteristic polynomial $\chi(x)$. 
\end{thm}
\begin{proof} 
	Denote by $\bar{x}$ the image of $x\in \Z[x]$ in the quotient ring $R=\Z[x]/(\chi(x))$, so that $R$ is the free $\Z$-module 
	$R=\Z+\Z \bar{x}+\cdots +\Z \bar{x}^{n-1}$ with the ring structure determined by  $\chi(\bar{x})=0$. 
	
	Let $M$ be an $R$-module of rank $1$.  Considering $M\cong \Z^n$ as a free $\Z$-module, 
	the action of $\bar{x}$ corresponds to multiplication by some matrix $S\in \Mat_n(\Z)$ on column vectors in $\Z^n$. 
	We claim that this matrix has characteristic polynomial $\chi(x)$.   
	Indeed, by extending the scalars to $\Q$, we obtain a vector space $V=M\otimes_\Z \Q\cong \Q^n$ equipped with an action of $K=\Q[\bar{x}]$, where 
	$\bar{x}$ acts on $\Q^n$ by multiplication by the same $S$. Since $V$ is a free $K$-module of rank $1$, 
	after possibly choosing a different basis of $V$, this module is $K$ with the action of $\bar{x}$ being the ring multiplication. 
	This implies that  $S$ is conjugate to the companion matrix of $\chi(x)$, so $S$ has characteristic polynomial $\chi(x)$. 
	
Let $\cM(R)$ be the set of isomorphism classes of $R$-modules of rank $1$, and let $\cC(\chi)$ be the set 
	of $\Z$-conjugacy classes of matrices in $\Mat_n(\Z)$ with characteristic polynomial $\chi(x)$. 
	Let $M_1$ and $M_2$ be two $R$-modules of rank $1$. 
	By fixing a basis of $M_1$ as a free $\Z$-module,  we obtain an isomorphism 
	$\iota_1\colon M_1\overset{\sim}{\to}\Z^n$. By the previous paragraph, the multiplication by $\bar{x}$ 
	on $M_1$ corresponds to multiplication by a matrix $S_1\in \Mat_n(\Z)$ on $\Z^n$ with characteristic polynomial $\chi(x)$. 
	Similarly, fixing a basis of $M_2$ as a $\Z$-module, i.e., an 
	isomorphism $\iota_2\colon M_2\overset{\sim}{\to} \Z^n$, the multiplication by $\bar{x}$ on $M_2$  becomes the multiplication by a matrix $S_2\in  \Mat_n(\Z)$ on $\Z^n$. 
	Having an $R$-module isomorphism $M_1\to M_2$ is equivalent to having a $\Z$-module isomorphism $\theta\colon M_1\to M_2$ such that 
	all the squares of the diagram below are commutative: 
	$$
	\xymatrix{ \Z^n \ar[d]_-{S_1} & M_1 \ar[l]_-{\iota_1} \ar[r]^-\theta \ar[d]_-{\bar{x}} & M_2 \ar[r]^-{\iota_2} \ar[d]^-{\bar{x}}  & \Z^n \ar[d]^-{S_2}  \\
		\Z^n & M_1 \ar[l]_-{\iota_1} \ar[r]^-\theta & M_2 \ar[r]^-{\iota_2}  & \Z^n}
	$$
	The composition $\iota_2\theta\iota_1^{-1}\colon \Z^n\to \Z^n$ corresponds to multiplication by an invertible matrix $T\in \GL_n(\Z)$ such that 
	$TS_1=S_2T$. We conclude that $M_1$ and $M_2$ are isomorphic $R$-modules if and only if $S_1$ and $S_2$ 
	are $\Z$-conjugate. Thus, we have an injective map $\cM(R)\to \cC(R)$ induced by $M\mapsto (\Z^n, S)$. 
	Now suppose $S\in \Mat_n(\Z)$ has characteristic polynomial $\chi(x)$. 
	Because $\chi(x)$ is separable, the minimal polynomial of $S$ is also $\chi(x)$, so the $\Z$-algebra $\Z[S]$ is isomorphic to $R$. Thus, 
	$\Z^n$ acquires a structure of an $R$-module of rank $1$, with $\bar{x}$ acting as $S$. We conclude that the map $\cM(R)\to  \cC(R)$ is also surjective.  
	Finally, because there is a bijection between $\ICM(R)$ and $\cM(R)$ by Lemma \ref{lemModICM}, the theorem follows. 
\end{proof}

%-------------------------------------------------------

\section{Ideal Class Monoid} \label{sICM} 

As in the introduction, let $n\geq 2$ be an integer, let  $a_1<a_2<\cdots<a_n$ be distinct integers, and put 
$$
\chi(x)= (x-a_1)(x-a_2)\cdots (x-a_n) \in \Z[x]. 
$$
Let $R=\Z[x]/(\chi(x))$. We want to compute the order of $\ICM(R)$. 
By Theorem \ref{thmLM}, it is enough to compute the number of $\Z$-conjugacy classes of matrices in $\Mat_n(\Z)$ with 
characteristic polynomial $\chi(x)$. To carry out this computation, we need some preliminary lemmas:

\begin{lem}\label{lemBasis}
	If the greatest common divisor of $c_1, \dots, c_n\in \Z$ is $1$, then the column vector $v=(c_1, \dots, c_n)$ 
	can be taken as part of a $\Z$-basis of $\Z^n$. 
\end{lem}
\begin{proof} The quotient $\Z^n/\Z v$ is a torsion-free $\Z$-module since $\gcd(c_1, \dots, c_n)=1$. Hence, $\Z^n/\Z v\cong \Z^{n-1}$ is a free 
	$\Z$-module of rank $n-1$. Since free modules are projective, the exact sequence $0\to \Z v \to \Z^n\to \Z^{n-1}\to 0$ splits (see \cite[p. 389]{DF}), so 
	$\Z^n\cong \Z v \oplus \Z^{n-1}$. Thus, $v$ is part of a basis of $\Z^n$. 
\end{proof}

\begin{lem}\label{lemUpTri}
If $A\in \Mat_n(\Z)$ is a matrix with characteristic polynomial $\chi_A(x)=\chi(x)$, then $A$ is conjugate over $\Z$ 
to an upper triangular matrix of the form 
$$
\begin{bmatrix} 
	a_1 & *  & \dots & * \\ 
	 & a_2  &\dots & * \\ 
	  & & \ddots & \vdots\\ 
	 & & & a_n 
\end{bmatrix}. 
$$
\end{lem}
\begin{proof} We use induction on $n$.  
	Consider $A$ as a linear transformation of $V=\Q^n$. Because $a_1\in \Q$ is an eigenvalue of $A$, there is $0\neq v\in \Q^n$ such that 
	$Av=a_1 v$. After scaling $v$, we may assume that the coordinates of $v=(c_1, \dots, c_n)$ are in $\Z$ and their gcd is $1$. Then, by Lemma \ref{lemBasis}, 
	we can choose a basis $v_1, \dots, v_n$ of $\Z^n$ such that $v=v_1$. With respect to this basis, the linear transformation corresponding to $A$ 
	is given by a matrix of the form 
	$
	A'= \begin{bmatrix} 
		a_1 & *   \\ 
		& A_0 
	\end{bmatrix},$
	where $A_0\in \Mat_{n-1}(\Z)$. 
	The matrix $A$ is conjugate to $A'$ over $\Z$. If $n=2$, this proves the claim of the lemma. If $n\geq 3$, then note that 
the characteristic polynomial of $A'$ is $(x-a_2)\cdots(x-a_n)$. 
	By the induction hypothesis, we can find $B\in \GL_{n-1}(\Z)$ 
	such that $B A_0 B^{-1}$ is upper-triangular, with $a_2, \dots, a_n$ on the main diagonal. 
	Finally, note that $\begin{bmatrix} 1 & \\ &  B\end{bmatrix} A' \begin{bmatrix} 1 & \\ &  B\end{bmatrix}^{-1}$ 
	has the required form. 
\end{proof}

\begin{notn}
Denote the set of upper-triangular matrices in $\Mat_n(\Z)$ of the form given in Lemma \ref{lemUpTri} by $\Omega$. 
Let $\Omega_0\subset \Omega$ be the subset of upper-triangular matrices $(c_{ij})$ such that $0\leq c_{ij}<a_j-a_i$ for all $1\leq i<j\leq n$. 
Note that $\Omega_0$ is finite and 
$$
\abs{\Omega_0} = \Delta= \prod_{i<j}(a_j-a_i). 
$$
Let $G_n$ be the subgroup of $\GL_n(\Z)$ consisting of upper-triangular matrices. Let $U_n \lhd G_n$ be the normal subgroup 
of unipotent matrices, i.e., those matrices which have $1$'s on the main diagonal. Let $D_n\subset G_n$ be the subgroup consisting 
of diagonal matrices. Since the elements of $D_n$ are invertible over $\Z$, we have 
$$
D_n=\{\diag(\pm 1, \dots, \pm 1)\} \cong (\Z/2\Z)^n. 
$$
Note that $G_n=D_nU_n$ and $U_n\cap D_n= \{I_n\}$, so $G_n$ is isomorphic to a semi-direct product $G_n\cong U_n\rtimes D_n$. 
\end{notn}
\begin{lem}
	If $A, B\in \Omega$ and $P\in \GL_n(\Z)$ are such that $PAP^{-1}=B$, then $P\in G_n$. 
\end{lem}
\begin{proof}
Let $e_1, e_2, \dots, e_n$ denote the standard basis of $\Z^n$. The conjugation can be written as $PA=BP$. Applying both sides 
of this matrix equation to $e_1$, we get $PA e_1 = a_1Pe_1 = B(P e_1)$. Since $P$ is invertible, $Pe_1\neq 0$ and $Pe_1$ is an eigenvector 
of $B$ with eigenvalue $a_1$. Since all $a_i$'s are distinct, the eigenspace of $B$ for $a_1$ is one-dimensional. Hence $Pe_1$ 
is a multiple of $e_1$, i.e., $Pe_1 = x_1 e_1$. Since $P$ is an integral matrix, $x_1\in \Z$ and $P$ has the form 
$$
P = \begin{bmatrix} x_1 & *   \\ &   P_0 \end{bmatrix},
$$  
where $P_0\in \GL_{n-1}(\Z)$. Now writing $PA = BP$ as 
$$
\begin{bmatrix} x_1 & *   \\ &   P_0 \end{bmatrix} \begin{bmatrix} a_1 & *   \\ &   A_0 \end{bmatrix} =
 \begin{bmatrix} a_1 & *   \\ &   B_0 \end{bmatrix} \begin{bmatrix} x_1 & *   \\ &   P_0 \end{bmatrix}
$$
we see that $P_0A_0 = B_0 P_0$. Hence the claim of the lemma can be deduced by induction on $n$. 
\end{proof}

\begin{prop} \label{propOrbits} \hfill
	\begin{enumerate}
	\item Every matrix in $\Omega$ is conjugate under $U_n$ to a matrix in $\Omega_0$. 
	\item No two matrices in $\Omega_0$ are conjugate under $U_n$. 
	\item The action of $U_n$ on $\Omega$ (by conjugation) has exactly $\Delta$ orbits. 
	\end{enumerate}
\end{prop}
\begin{proof}
(1)	For fixed $1\leq i<j\leq n$ and $t\in \Z$, consider the matrix $P=I_n+tE_{ij}\in U_n$, where $E_{ij}$ denotes the 
	$n\times n$-matrix with the $(i, j)$-th entry equal to $1$ and all other entries equal to $0$. Note that $P^{-1}=I_n-tE_{ij}$. 
	Given $A=(a_{sk})\in \Omega$, the matrix $A'=(a_{sk}')=PAP^{-1}$ is the matrix that one obtains by adding $t$ times $j$-th row of $A$ to its $i$-th row 
	and adding $(-t)$ times $i$-th column of $A$ to the $j$-th column. Hence $a_{ij}'=a_{ij}+t(a_j-a_i)$. Moreover, because $P$ 
	is upper-triangular, in $A'$ only the entries above $a_{ij}'$ (in the same column) and to the right of $a_{ij}'$ (in the same row) 
	can be different from the entries of $A$.  Thus, by a sequence of conjugations by appropriate $P$'s, we can change the entries of the 
	$(n-1)$-th row working from left to right, then the entries of the $(n-2)$-th row of the resulting matrix from left to right and so on, eventually 
	obtaining a matrix in $\Omega_0$.  
	
	(2) Suppose $PA=BP$ for some $P=(x_{ij})\in U_n$ and $A=(a_{ij}), B=(b_{ij})\in \Omega_0$. 
	For $n=2$, we have a matrix equation 
	$$\begin{bmatrix} 1 & t \\ & 1\end{bmatrix}\begin{bmatrix} a_1 & a_{12} \\ & a_2\end{bmatrix} =
	\begin{bmatrix} a_1 & b_{12} \\ & a_2\end{bmatrix}\begin{bmatrix} 1 & t \\ & 1\end{bmatrix},
	$$
	which leads to $b_{12}-a_{12}=t(a_2-a_1)$. Since $\abs{b_{12}-a_{12}}< (a_2-a_1)$, we must have $t=0$. For $n\geq 3$, 
	we will use induction to show that $P=I_n$. 
	Let $A_0, B_0, P_0\in U_{n-1}$ denote the upper left submatrices of $A$, $B$ and $P$, respectively. Then  we have $P_0A_0=B_0P_0$. 
	Moreover, the matrices $A_0$ and $B_0$ are in the set analogous to $\Omega_0$ that arises from the polynomial $(x-a_1)\cdots (x-a_{n-1})$. 
	Thus, by induction $P_0=I_{n-1}$. The same argument works for the lower right $(n-1)\times(n-1)$ submatrices of $A, B, P$. 
	It remains to show that $x_{1n}=0$. For this we consider $(1, n)$-th entries of both sides of the matrix equation $PA=BP$: 
	$$
	a_{1n}+a_n x_{1n}=a_1x_{1n}+b_{1n}.
	$$
	The argument that we used for $n=2$ works here too, so we get $x_{1n}=0$. Therefore, $P=I_n$ and $A=B$. 
	
	(3) This immediately follows from (1) and (2). 
\end{proof}

Let $\oOm\colonequals \Omega/U_n$ be the set of orbits of the action of $U_n$ on $\Omega$ by conjugation. 
We proved that $\abs{\oOm} = \abs{\Omega_0}=\Delta$. Since $U_n$ is normal, the group $D_n$ acts on $\oOm$, i.e., $D_n$ 
acts on the set of orbits of $U_n$. We have reduced the problem of computing $\abs{\ICM(R)}$ to computing the number of orbits 
of $D_n$ acting on the finite set $\oOm$. More specifically, by Burnside's Lemma, 
\begin{equation}\label{eqBurnside}
\abs{\ICM(R)} = \frac{1}{2^n}\sum_{P\in D_n} \abs{\Fix(P)},
\end{equation}
where $2^n=\abs{D_n}$ and $\Fix(P)$ is the set of fixed points of $P$ acting on $\oOm$.

Recall that for a positive rational number $\alpha$, the ``floor" function $\lfloor \alpha\rfloor$ is the largest integer $m$ such that $m\leq \alpha$. 

\begin{lem}\label{exampleICM2}
If $n=2$, then 
$$\abs{\ICM(R)} = \left\lfloor \frac{a_2-a_1}{2} \right\rfloor+1. 
$$
\end{lem}
\begin{proof}
When $n=2$, 
	$$\Omega_0 = \left\{\begin{bmatrix} a_1 & b \\ & a_2\end{bmatrix}\ \bigg|\  0\leq b< a_2-a_1\right\}
	$$  
	and $D_2=\{\pm I_2, \pm P\}$, where $P=\diag(-1, 1)$. Obviously, $\Fix(\pm I_2)=\oOm$ and $\Fix(P)=\Fix(-P)$. Thus, 
	it remains to determine when $\begin{bmatrix} a_1 & b \\ & a_2\end{bmatrix}\in \Omega_0$ and $P\begin{bmatrix} a_1 & b \\ & a_2\end{bmatrix}P^{-1}=\begin{bmatrix} a_1 & - b \\ & a_2\end{bmatrix}$  are in the same $U_2$ orbit.  As one easily calculates, these 
			two matrices are in the same $U_2$ orbit if and only if there is $t\in \Z$ such that $b=-b+t(a_2-a_1)$. Unless $b=0$, this is equivalent 
			to $(a_2-a_1)$ being even and $b=(a_2-a_1)/2$. Thus, 
			$$
			\abs{\Fix(P)}=
			\begin{cases}
				1, & \text{if $a_2-a_1$ is odd}; \\
				2, &  \text{if $a_2-a_1$ is even}. 
			\end{cases}
			$$
			By \eqref{eqBurnside}, 
			\[\abs{\ICM(R)} = \frac{1}{4}\left(2(a_2-a_1)+2\abs{\Fix(P)}\right) =\lfloor (a_2-a_1)/2 \rfloor+1. \qedhere
			\]
\end{proof}

Given integers $m_1, \dots, m_t$, put 
$$
\ind_{m_1, \dots, m_t} =
\begin{cases}
1, & \text{if all $m_1, \dots, m_t$ are even}; \\ 
0, & \text{otherwise}. 
\end{cases}
$$
%In the next two examples for the a statement $\mathcal{P}(a_1, a_2, \ldots, a_n)$ about the numbers $a_1, a_2, \ldots, a_n$ we will denote by $\ind_{\mathcal{P}(a_1, a_2, \ldots, a_n)}$ the function that gives $1$ when $\mathcal{P}(a_1, a_2, \ldots, a_n)$ is true and $0$, when $\mathcal{P}(a_1, a_2, \ldots, a_n)$ is false. 

\begin{prop}\label{exampleICM3}
    If $n=3$, then 
    \begin{align*}
        \abs{\ICM(R)} = \frac{1}{4} \Big( \Delta &+ (a_3 - a_2)(1 + \ind_{a_2 - a_1} + \ind_{a_3 - a_1}) \\
        &+ (a_3 - a_1)(1 + \ind_{a_2 - a_1} + \ind_{a_3 - a_2})\\
        &+ (a_2 - a_1)(1 + \ind_{a_3 - a_1} + \ind_{a_3 - a_2}) \Big).
    \end{align*}
\end{prop}
\begin{proof}
We denote $P_1 = \diag(-1, 1, 1)$, $P_2 = \diag(1, -1, 1)$, and $P_3 = \diag(1, 1, -1)$. Note that for 
$P \in D_3$, $\abs{\Fix(-P)} = \abs{\Fix(P)}$ and $\abs{\Fix{(I_3) }} = \abs{\oOm} = \Delta$. So by \eqref{eqBurnside} we have
    \begin{equation}\label{eq-ICM3}
        \abs{\ICM (R)} = \frac{\Delta + \abs{\Fix{(P_1)}} + \abs{\Fix{(P_2)}} + \abs{\Fix{(P_3)}}}{4}. 
    \end{equation}
    To find $|\Fix(P_1)|$, take a matrix $A \in \Omega_0$ and assume that conjugation by $P_1$ fixes the class of $A$ in $\oOm$. Observe that, if 
    \[
        A =
        \begin{bmatrix}
            a_1 & x & y\\
            & a_2 & z\\
            & & a_3
        \end{bmatrix},
    \]
    then 
    \[
        P_1AP_1^{-1} =
        \begin{bmatrix}
            a_1 & -x & -y\\
            & a_2 & z\\
            & & a_3
        \end{bmatrix}.
    \]
    Suppose $x = 0$. If $y = 0$, then conjugation by $P_1$ fixes the class of $A$. If $y \neq 0$, we note that conjugating $P_1AP_1^{-1}$ by the matrix $I + E_{13}$ only increases $(1, 3)$-th entry of $P_1AP_1^{-1}$ by $a_3 - a_1$. The matrix obtained after this conjugation belongs to $\Omega_0$ and it belongs to the same class as $A$ if and only if $2y = a_3 - a_1$. So there are $(a_3 - a_2)(1 + \ind_{a_3 - a_1})$ classes of matrices in $\oOm$, fixed by $P_1$ and such that their representative $A \in \Omega_0$ satisfies $x = 0$. 

    Now suppose $x \neq 0$. Note that the congruence class of $(1, 2)$-th entry of $A$ modulo $a_2 - a_1$ doesn't change when conjugating $A$ with any element of $U_3$. Hence for $P_1$ to fix the class of $A$, we must have $x \equiv - x \pmod{a_2 - a_1}$. Since $x \neq 0$, the later is equivalent to $x = \frac{a_2 - a_1}{2}$, which is possible only when $2 \mid a_2 - a_1$. Assume that $2 \mid a_2 - a_1$ and observe that conjugating $P_1AP_1^{-1}$ by the matrix $I + E_{12}$ we get
    \[
        (I + E_{12})P_1AP_1^{-1}(I + E_{12})^{-1} = 
        \begin{bmatrix}
            a_1 & (a_2 - a_1) - x & z - y\\
            & a_2 & z\\
            & & a_3
        \end{bmatrix}.
    \]
    By the same argument as in the case of $x = 0$, we see that $P_1$ fixes the class of $A$ if and only if $z - y \equiv y \pmod{a_3 - a_1}$. The later is equivalent to $(a_3 - a_1) \mid (2y - z)$. If $2 \nmid a_3 - a_1$, then for each value of $z$ with $0 \le z < a_3 - a_2$, there is a unique value of $y$ with $0 \le y < a_3 - a_1$ such that $a_3 - a_1 \mid 2y - z$. If $2 \mid a_3 - a_1$, then $a_3 - a_1 \mid 2y - z$ implies $2 \mid z$. Then for any even value of $z$ with $0 \le z < a_3 - a_2$, there are exactly two values of $y$ with $0 \le y < a_3 - a_1$ such that $a_3 - a_1 \mid 2y - z$. In this case $a_3 - a_2$ must be even, so there are $\frac{a_3 - a_2}{2}$ even values of $z$ with $0 \le z < a_3 - a_1$. In both cases we have $a_3 - a_2$ such pairs $(y, z)$. Therefore,
    \[
        \abs{\Fix(P_1)} = (a_3 - a_2)(1 + \ind_{a_2 - a_1} + \ind_{a_3 - a_1}). 
    \]
    Similar calculations for $P_2$ and $P_3$ give
    \begin{align*}
        \abs{\Fix(P_2)} &= (a_3 - a_1)(1 + \ind_{a_2 - a_1} + \ind_{a_3 - a_2}),\\
        \abs{\Fix(P_3)} &= (a_2 - a_1)(1 + \ind_{a_3 - a_1} + \ind_{a_3 - a_2}). 
    \end{align*}
    The claim of the proposition follows by plugging everything into \eqref{eq-ICM3}. 
\end{proof}

\begin{rem} 
For $n = 4$, computations with Python suggest the following (conjectural) formula for $\abs{ \ICM (R)}$:
\begin{align*}
  \abs{\ICM(R)} = \frac{1}{8} \Big ( \Delta 
    &+ (1 + \ind_{a_2 - a_1} + \ind_{a_3 - a_1} + \ind_{a_4 - a_1})(a_3 - a_2)(a_4 - a_2)(a_4 - a_3) \\
    &+ (1 + \ind_{a_2 - a_1} + \ind_{a_3 - a_2} + \ind_{a_4 - a_2})(a_3 - a_1)(a_4 - a_1)(a_4 - a_3) \\
    &+ (1 + \ind_{a_3 - a_1} + \ind_{a_3 - a_2} + \ind_{a_4 - a_3})(a_2 - a_1)(a_4 - a_1)(a_4 - a_2) \\
    &+ (1 + \ind_{a_4 - a_1} + \ind_{a_4 - a_2} + \ind_{a_4 - a_3})(a_2 - a_1)(a_3 - a_2)(a_3 - a_1) \\
    &+ (1 + \ind_{a_3 - a_2} + \ind_{a_3 - a_1} + \ind_{a_4 - a_2} + \ind_{a_4 - a_1} \\
    &+ \ind_{a_4 - a_1, a_3 - a_2} + \ind_{a_4 - a_2, a_3 - a_1})(a_2 - a_1)(a_4 - a_3) \\
    &+ (1 + \ind_{a_2 - a_1} + \ind_{a_4 - a_1} + \ind_{a_3 - a_2} + \ind_{a_4 - a_3} \\
    &+ \ind_{a_2 - a_1, a_4 - a_3} + \ind_{a_4 - a_1, a_3 - a_2})(a_3 - a_1)(a_4 - a_2) \\
    &+ (1 + \ind_{a_2 - a_1} + \ind_{a_3 - a_1} + \ind_{a_4 - a_2} + \ind_{a_4 - a_3} \\
    &+ \ind_{a_4 - a_2, a_3 - a_1} + \ind_{a_2 - a_1, a_4 - a_3})(a_4 - a_1)(a_3 - a_2) \Big ). 
\end{align*}
\end{rem}

%End of the Part Edited by Ruben

From the previous examples it is not hard to see that $\abs{\ICM(R)} \sim \Delta/2^{n-1}$ as $\Delta\to \infty$. Next, we prove that this 
is a general phenomenon. 

\begin{thm} For any fixed $n\geq 2$, 
$$
\lim_{\Delta\to \infty}\frac{\abs{\ICM(R)} \cdot 2^{n-1}}{\Delta} = 1. 
$$	
\end{thm}
\begin{proof} 
	Since $\Fix(\pm I_n)=\oOm$, we can rewrite \eqref{eqBurnside} as 
	$$
	\frac{2^{n-1}}{\Delta}\cdot \abs{\ICM(R)} = 1+\frac{1}{2\Delta} \sum_{P\in D_n\bs \{ \pm I_n \} } \abs{\Fix(P)}. 
	$$
    We will show that for any $P\in D_n$ such that $P\neq \pm I_n$, we have 
	$$
	\lim_{\Delta\to \infty} \frac{ \abs{\Fix(P)}}{\Delta} =0. 
	$$
    
    For a subset of indices $J \subseteq \{1, 2, \dots, n\}$, we denote $P_J = \diag(x_1, x_2, \dots, x_n)$ where $x_j = -1$ if $j \in J$ and $x_j = 1$ if $j \not \in J$. We claim that 
    \begin{equation*}
        \abs{\Fix(P_J)} \le 2^{\abs{J} (n - \abs{J})} \prod_{\substack{1 \le i < j \le n \\ i, j \in J}}(a_j - a_i) \prod_{\substack{1 \le i < j \le n \\ i, j \not \in J}}(a_j - a_i).
    \end{equation*}
    Suppose that the conjugation by $P_J$ fixes the orbit of the matrix $A = (a_{ij})\in \Omega_0$. We note that the conjugation by $P_J$ changes the signs of those entries $a_{ij}$ of $A$ for which exactly one of the indices $i$ and $j$ belongs to $J$. There are exactly $\abs{J}(n - \abs{J})$ such entries that lie above the diagonal of $A$. Next, note that there are exactly 
    $$\prod_{\substack{1 \le i < j \le n \\ i, j \in J}}(a_j - a_i) \prod_{\substack{1 \le i < j \le n \\ i, j \not \in J}}(a_j - a_i)$$ possible combinations for the rest of the entries of $A$. Now we apply the process described in the proof of the Proposition \ref{propOrbits} to bring $P_JAP_J^{-1}$ to a matrix from $\Omega_0$ by 
    conjugating with elements of $U_n$. We observe that each time when we need to change the $(i, j)$-th entry with exactly one of the indices $i$ and $j$ in $J$, $a_{ij}$ must satisfy a congruence $-a_{ij} + C \equiv a_{ij} \pmod{a_j - a_i}$, for some already known integer $C$. Since this congruence can have at most $2$ solutions with $0 \le a_{ij} < a_j - a_i$ we obtain the required upper bound for $\abs{\Fix(P_J)}$. 

    Now for any $P = \diag(x_1, x_2, \dots, x_n) \in D_n \bs \{\pm I_n\}$, we consider the corresponding set of indices $J$. Note that $J \neq \varnothing $ and $\{1, 2, \dots, n\} \bs J \neq \varnothing$. If $x_1 \neq x_n$, then note that the proved inequality implies $\frac{\abs{\Fix(P)}}{\Delta} \le \frac{2^{\abs{J}(n - \abs{J})}}{a_n - a_1}$. Since 
    $(a_n - a_1) \rightarrow \infty$ when $\Delta \rightarrow \infty$, $\lim_{\Delta\to \infty} \frac{ \abs{\Fix(P)}}{\Delta} =0$. If $x_1 = x_n$, since $J \neq \varnothing$ and $\{1, 2, \dots, n\} \bs J \neq \varnothing$, $n \ge 3$ and there is an index $j \in \{1, 2, \dots, n\}$ such that $x_j = -x_1 = -x_n$. Therefore $\frac{\abs{\Fix(P)}}{\Delta} \le \frac{2^{\abs{J}(n - \abs{J})}}{(a_n - a_j)(a_j - a_1)}$ and hence $\lim_{\Delta\to \infty} \frac{ \abs{\Fix(P)}}{\Delta} =0$.
\end{proof}

%===================================

\section{Ideal class group} \label{sICG} 

As in Section \ref{sICM}, let $R=\Z[x]/(\chi(x))$ and $K=\Q[x]/(\chi(x))$, 
where $$\chi(x)=(x-a_1)\cdots(x-a_n)\in \Z[x], \quad a_1<a_2<\cdots <a_n. $$ 
In this section we compute the order of $\Cl(R)$ using the Class Number Formula for orders in products 
of number fields proved in \cite{JP}. 

First, we realize $R$ as a subring of $\Z^n$, which helps to determine its group of units. Consider the evaluation homomorphism 
\begin{align*} 
	\phi\colon \Z[x] & \To  \Z^n,  \\ 
	f(x) &\longmapsto (f(a_1), \dots, f(a_n)). 
\end{align*}
Since $\phi(\chi)=(0, \dots, 0)$, the above map factors through a homomorphism $R\to \Z^n$, which we again denote by $\phi$. 

\begin{lem}\label{lemR'} The homomorphism $\phi\colon R\to \Z^n$ is injective and its image is a subring of 
	$$
	R'=\{(b_1, b_2, \dots, b_n)\in \Z^n \mid b_j\equiv b_i\Mod{(a_j-a_i)}\text{ for all }1\leq i<j\leq n\}. 
	$$
\end{lem}
\begin{proof} Let $f(x)\in \Z[x]$. If $\phi(f)=0$, then $a_1, \dots, a_n$ are roots of $f(x)$. Since $\chi(x)$ 
	is monic, we can apply the division algorithm to write $f(x)=g(x)\chi(x)+r(x)$, where $g(x), f(x)\in \Z[x]$ and either $\deg(r)<\deg(\chi)=n$ or $r=0$. 
	Evaluating both sides of the above equation at the $a_i$'s, we see that $r(a_i)=0$ for all $1\leq i\leq n$. Thus,  $r=0$, as otherwise $\deg(r)\geq n$. 
	We conclude that the kernel of $\phi\colon \Z[x]\to \Z^n$ is $(\chi(x))$, so $\phi\colon R\to \Z^n$ is injective.  
	If  $(b_1, \dots, b_n)\in \phi(R)$, then there exists 
$f(x)\in \Z[x]$ such that $f(a_i)=b_i$ for all $1\leq i\leq n$. Since 
$f(a_j)-f(a_i)$ is divisible by $(a_j-a_i)$, we see that $(b_1, \dots, b_n) \in R'$. Thus, $\phi\colon R\hookrightarrow R'$. 
\end{proof}

\begin{rem}\label{remR2}
If $n=2$, then $\phi(R)=R'$. To see this suppose $(b_1, b_2)\in \Z^2$ is such that $(a_2-a_1)\mid (b_2-b_1)$. 
	Let $g(x)=x+(b_1-a_1)$, so that $g(a_1)=b_1$. Write $b_2=b_1+(a_2-a_1)t$ with $t\in \Z$. If we take $f(x)=g(x)+(x-a_1)(t-1)$, then 
	$f(a_1)=g(a_1)=b_1$ and $f(a_2)=g(a_2)+(a_2-a_1)(t-1)=(a_2-a_1)+b_1+(a_2-a_1)(t-1)=b_2$. Hence $\phi(f)=(b_1, b_2)$. 
	
	Generally, for $n\geq 3$, the inclusion $\phi(R)\subseteq R'$ can be strict. This is related to the fact that the polynomials arising from the 
	Lagrange interpolation do not have integer coefficients. As an explicit example, let $a_1=0, a_2=2, a_3=4$ and $b_1=0, b_2=2, b_3=8$. 
	We claim that there is no $f(x)\in \Z[x]$ such that $f(a_i)=b_i$ for $i=1,2,3$. Taking the residue of the hypothetical $f(x)$ by $\chi(x)$, 
	we may assume that $f(x)=ax^2+bx+c$. Now the required equalities force $b=c=0$ and $2a=1$. 
\end{rem}

% \begin{cor}\label{cor3.3}
% 	If $\Delta$ is large enough, or $n\geq 6$, then $R^\times=\pm 1$. 
% \end{cor}
% \begin{proof} By Lemma \ref{lemR'}, it is enough to show that under the given conditions $(R')^\times =\pm 1$. 
% 	We have $(\Z^n)^\times = \{(\eps_1, \dots, \eps_n)\mid \eps_i=\pm 1, 1\leq i\leq n\}$. Suppose $\eps\in (\Z^n)^\times $ 
% 	lies in $(R')^\times$, but $\eps_j\neq \eps_1$ for some $1< i\leq n$. If $n\geq 6$ or $\Delta$ is sufficiently large, then $a_n-a_1\geq 5$. 
% 	Since $\eps_1\equiv \eps_n\Mod{a_n-a_1}$, we have $\eps_1=\eps_n$. But now $\eps_1\equiv \eps_j\Mod{a_j-a_1}$ implies 
% 	$a_j-a_1\leq 2$ and similarly $\eps_j\equiv \eps_n\Mod{a_n-a_j}$ implies $a_n-a_j\leq 2$. Thus, $a_n-a_1\leq 4$, a contradiction. 
% \end{proof}

\begin{cor}\label{cor3.3}
	If $a_n - a_1 \geq 4$, then $R^\times=\{\pm 1\}$. 
\end{cor}
\begin{proof} By Lemma \ref{lemR'}, it is enough to show that under the given conditions $(\im{\phi})^\times = \{ \pm 1 \}$. 
	We have $(\Z^n)^\times = \{(\eps_1, \dots, \eps_n)\mid \eps_i=\pm 1, 1\leq i\leq n\}$. Suppose $\eps\in (\Z^n)^\times $ 
	lies in $(\im{\phi})^\times$, but $\eps_j\neq \eps_1$ for some $1< i\leq n$. Since $\eps_1\equiv \eps_n\Mod{a_n-a_1}$, we have $\eps_1=\eps_n$. Since $\varepsilon \in \im{\phi}$, there is a polynomial $f \in \Z[x]$, such that $f(a_i) = \varepsilon_i$, for $1 \leq i \leq n$. Note that $(x - a_1)(x - a_n)  \mid f(x) - \varepsilon_1$. Hence there is a polynomial $q \in \Z[x]$, such that $f(x) - \varepsilon_1 = (x - a_1)(x - a_n)q(x)$. Substituting $x = a_j$ in this equation we obtain $\varepsilon_j - \varepsilon_1 = (a_j - a_1)(a_j - a_n)q(a_j)$. Since $\varepsilon_j \neq \varepsilon_1$, we must have $2 \geq (a_j - a_1)(a_n - a_j)$ which is not possible when $a_n - a_1 \geq 4$. 
\end{proof}

\begin{rem} \label{unit_rem}
    Note that if $n\geq 5$ or $\Delta$ is sufficiently large, then $a_n-a_1\geq 4$, and hence by the corollary above $\abs{R^\times} = 2$. Thus, there can be only finitely many cases where $\abs{R^\times} > 2$. Here is the list of these cases, which is easy to deduce by direct calculations:
    \begin{itemize}
        \item If $n = 2$ and $a_2 - a_1 \le 2$, then $R^\times \cong \{(b_1, b_2) \mid b_1, b_2 \in \{\pm 1\}\}$. 
        \item If $n = 3$ and $a_3 - a_1 = 2$, then $R^\times \cong \{(b_1, b_2, b_3) \mid b_1, b_2, b_3 \in \{\pm 1\}\}$.
        \item If $n = 3$ and $a_3 - a_1 = 3$, then $R^\times \cong \{(b_1, b_2, b_3) \mid b_1, b_2, b_3 \in \{\pm 1\} \text{ and } b_1 = b_3\}$.
        \item If $n = 4$ and $a_4 - a_1 = 3$, then $R^\times \cong \{(b_1, b_2, b_3, b_4) \mid b_1, b_2, b_3, b_4 \in \{\pm 1\} \text{ and } b_1 = b_4\}$. 
    \end{itemize}
    We briefly discuss the last case. Without loss of generality, assume that $a_1 = 0$, $a_2 = 1$, $a_3 = 2$, and $a_4 = 3$. Consider the polynomials $r(x) = 1 - x(x - 2)(x - 3)$ and $s(x) = 1 + x(x - 1)(x - 3)$. Then $\phi(r) = (1, -1, 1, 1)$ and $\phi(s) = (1, 1, -1, 1)$. On one hand, $(1, 1, 1, 1), (-1, -1, -1, -1) \in \im \phi$ and if $a \in \im \phi$, then $-a \in \im \phi$, so $|\im{\phi}^\times| \ge 8$. On the other hand, if $(b_1, b_2, b_3, b_4) \in \im{\phi}^\times$, then there is a polynomial $f \in \Z[x]$ such that $f(a_j) = b_j$, $j = 1, 2, 3, 4$. This implies $3 = a_4 - a_1 \mid f(a_4) - f(a_1) = b_4 - b_1$. Since $b_1, b_4 \in \{\pm 1\}$, we must have $b_1 = b_4$. Therefore, 
    \[
        R^\times \cong \im{\phi}^\times = \{(b_1, b_2, b_3, b_4) \mid b_1, b_2, b_3, b_4 \in \{\pm 1\} \text{ and } b_1 = b_4\}.
    \] 
\end{rem}

The \textit{zeta function} $\zeta_R(s)$ of $R$ is defined as 
$$
\zeta_R(s)=\prod_{\fp}\left(1-N(\fp)^{-s}\right)^{-1},
$$
where the product is over the maximal ideals of $R$ and $N(\fp)\colonequals \abs{R/\fp}$. The function $\zeta_R(s)$ converges for $\re(s)>1$ 
and has a pole of at $s=1$ of order $n$; see \cite[$\S$5.6]{JP}. 

Let $\bar{x}$ denote the image of $x$ in $R$ under the quotient map $\Z[x]\to R=\Z[x]/(\chi(x))$. Then $\{1, \bar{x}, \dots, \bar{x}^{n-1}\}$ 
is a basis of $R$ as a $\Z$-module. Therefore, the \textit{discriminant} of $R$ is 
\begin{align*}
\disc(R) & = \det (\Tr_{K/\Q}(\bar{x}^i\bar{x}^j))_{1\leq i, j\leq n} = \Delta^2. 
\end{align*}
Next, since $R^\times$ is finite, the regulator of $R$ is trivial. Thus, in our case, Theorem 1.1 in \cite{JP} specializes  
to the following 
\begin{equation}\label{eqCNF}
\lim_{s\to 1}(s-1)^n \zeta_R(s) = \frac{2^n}{\abs{R^\times}\cdot \Delta} \abs{\Cl(R)}. 
\end{equation}

Let $\fp\lhd R$ be a maximal ideal  and let $\fp'\lhd \Z^n$ be a maximal ideal of $\Z^n$ lying over $\fp$, i.e., $\fp=\fp'\cap R$. 
Since $\Z^n/\fp'\cong \F_p$ for some prime $p$ of $\Z$, we also have $R/\fp\cong \F_p$. The number of maximal 
ideals of $R$ with residue field $\F_p$ is finite; denote this number by $A(p)$. 
Let $\alpha_1, \dots, \alpha_k\in \F_p$ be the distinct residues 
that one obtains by reducing $a_1, \dots, a_n$  modulo $p$. Given $1\leq i\leq k$, let $m_i$ be the number of $a_j$'s such that $\alpha_i\equiv a_j \Mod{p}$. 
Then 
$$
\chi(x) \equiv (x-\alpha_1)^{m_1}\cdots (x-\alpha_k)^{m_k} \Mod{p},
$$
and $m_1+\cdots+m_k=n$. Now 
\begin{align*}
R/(p)=\Z[x]/(p, \chi(x))&\cong \F_p[x]/\left((x-\alpha_1)^{m_1}\cdots (x-\alpha_k)^{m_k}\right). 
\end{align*}
The maximal ideals of this last ring are the ideals generated by $(x-\alpha_i)$,  $1\leq i\leq k$. Thus, $A(p)=k$. We can rewrite 
\begin{align*}
\zeta_R(s) &= \prod_{k=1}^n \prod_{A(p)=k} (1-p^{-s})^{-k}  \\ &= \prod_{p\nmid \Delta} (1-p^{-s})^{-n}  \prod_{k=1}^{n-1} \prod_{A(p)=k} (1-p^{-s})^{-k} \\ 
&= \zeta_{\Z}(s)^n  \prod_{k=1}^{n-1} \prod_{A(p)=k} (1-p^{-s})^{n-k}.
\end{align*}
Therefore, 
\begin{align}\label{eq-limZ}
\lim_{s\to 1}(s-1)^n \zeta_R(s) & = \prod_{k=1}^{n-1} \prod_{A(p)=k} \left(1-p^{-1}\right)^{n-k} \\ 
\nonumber &= \prod_{l=1}^{n-1} \prod_{A(p)\leq l} \left(1-p^{-1}\right). 
\end{align}

\begin{lem}\label{lem-gcdDelta} Let $\Delta_l$ be defined as in \eqref{eqDeltak}. Then 
	$$
	\prod_{A(p)\leq l} \left(1-p^{-1}\right) = \varphi(\Delta_l)/\Delta_l. 
	$$
\end{lem}
\begin{proof}
	It is enough to prove that $p\mid \Delta_l$ if and only if $A(p)\leq l$. For this note that $A(p)\leq l$ is equivalent to the 
	statement that for any subset $S\subset \{1, \dots, n\}$ of order $l+1$ we have $a_i\equiv a_j\Mod{l}$ for at least two 
	distinct $i, j\in S$. Thus, 
	\[ 
	A(p)\leq l \quad \Longleftrightarrow  \quad p \mid \Delta_S \text{ for all }S\subset \{1, \dots, n\} \text{ with }\abs{S}=l+1 \quad  \Longleftrightarrow\quad 
	p\mid \Delta_l. \qedhere
	\]
\end{proof}

With the help of Lemma \ref{lem-gcdDelta}, we can rewrite \eqref{eq-limZ} as 

\begin{align}\label{eq-limZ2}
	\lim_{s\to 1}(s-1)^n \zeta_R(s) = \prod_{l=1}^{n-1} \frac{\varphi(\Delta_l)}{\Delta_l} =\frac{\varphi(\Delta)}{\Delta}\prod_{l=1}^{n-2} \frac{\varphi(\Delta_l)}{\Delta_l}. 
\end{align}

Now combining \eqref{eqCNF} and \eqref{eq-limZ2} with Corollary \ref{cor3.3} and Remark \ref{unit_rem}, we obtain the following: 

\begin{thm}\label{thmCl}
	If $a_n-a_1\geq 4$, then 
	$$
	\abs{\Cl(R)} = \frac{\varphi(\Delta)}{2^{n-1}} \prod_{l=1}^{n-2} \frac{\varphi(\Delta_l)}{\Delta_l}.
	$$
	If $a_n-a_1<4$, then $\abs{\Cl(R)} =1$. 
\end{thm}

\begin{lem} \label{prop_rho}
    For any integers $b_1, b_2, \ldots, b_m$,
    \[
        \prod_{1 \le i < j \le m} (j - i) \mid \prod_{1 \le i < j \le m} (b_j - b_i).
    \]
\end{lem}
\begin{proof}
    Note that $\prod_{1 \le i < j \le m} (j - i) = \prod_{k = 1}^{m - 1} k!$. Now, using the Vandermonde determinant formula and performing suitable row operations, we obtain
    \begin{align*}
        \prod_{1 \le i < j \le m} \frac{b_j - b_i}{j - i} &= 
        \prod_{k = 1}^{m - 1} \frac{1}{k!} \cdot \det
        \begin{bmatrix}
            1 & 1 & \dotsm & 1\\
            b_1 & b_2 & \dotsm & b_m\\
            b_1^2 & b_2^2 & \dotsm & b_m^2 \\
            \dotsm & \dotsm & \dotsm & \dotsm \\
            b_1^{m - 1} & b_2^{m - 1} & \dotsm & b_m^{m - 1}
        \end{bmatrix} \\
        &= \det
        \begin{bmatrix}
            1 & 1 & \dotsm & 1\\
            \binom{{b_1}}{1} & \binom{b_2}{1} & \dotsm & \binom{b_m}{1} \\
            \binom{b_1}{2} & \binom{b_2}{2} & \dotsm & \binom{b_m}{2} \\
            \dotsm & \dotsm & \dotsm & \dotsm \\
            \binom{b_1}{m - 1} & \binom{b_2}{m - 1} & \dotsm & \binom{b_m}{m - 1}
        \end{bmatrix},  
    \end{align*}
   where for any real $x$ and nonnegative integer $h$, $\binom{x}{h}$ denotes $\frac{x(x - 1) \dotsm (x - h + 1)}{h!}$. 
   Since the entries of the last matrix are integers, its determinant is also an integer, which concludes the proof of the lemma. 
\end{proof}

\begin{comment}
\begin{notn}
    For a positive integer $m$ denote
    \[
        \rho_m = \prod_{1 \le i < j \le m + 1} (j - i).
    \]
\end{notn}
\end{comment}

\begin{thm} 
\begin{align*}
    & \liminf_{\Delta\to \infty} \frac{\abs{\Cl(R)} \cdot 2^{n-1}}{\varphi(\Delta)} = 
    \begin{cases}
1, & \text{if } n=2; \\ 
0, & \text{if } n\geq 3.
    \end{cases} \\ 
    & \limsup_{\Delta\to \infty} \frac{\abs{\Cl(R)} \cdot 2^{n-1}}{\varphi(\Delta)} = \prod_{\substack{p \le n - 2 \\ p \text{ is prime}}} (1 - p^{-1})^{n - 1 - p}.
\end{align*}

\end{thm}
\begin{proof} For $n=2$, the claims immediately follow from Theorem \ref{thmCl}. 
    Now suppose $n\geq 3$. Then 
    $$\frac{\abs{\Cl(R)} \cdot 2^{n-1}}{\varphi(\Delta)}\leq \varphi(\Delta_1)/\Delta_1 =
    \prod_{p\mid \Delta_1}\left(1-p^{-1}\right). 
    $$
    This last product tends to $1/\zeta_\Z(1)=0$ if $\Delta_1$ is divisible by the product of primes $\prod_{p\leq N} p$ and $N\to \infty$. Let $m\geq 1$ be an arbitrary integer. Let $a_i=a_1+(i-1)m$, $1\leq i\leq n$. Since for this choice of $a_i$'s the difference $a_j-a_i$ is divisible by $m$ for all $1\leq i<j\leq n$, we see that $\Delta_1$ is divisible by $m$. This proves the claim about limit inferior.  
    
     For a positive integer $m$ denote
    \[
    \rho_m = \prod_{1 \le i < j \le m + 1} (j - i).
    \]
For $l = 1, 2, \ldots, n - 2$, if $S \subseteq \{1, 2, \ldots, n\}$ and $\abs{S} = l + 1$, then by Lemma \ref{prop_rho}, $\rho_l \mid \Delta_S$. Therefore $\rho_l \mid \Delta_l$, for $l = 1, 2, \ldots, n - 2$. Hence, if $a_n - a_1 \ge 4$, then
   \begin{align*}
        \frac{\abs{\Cl(R)} \cdot 2^{n - 1}}{\varphi(\Delta)} & = \prod_{l = 1}^{n - 2} \frac{\varphi(\Delta_l)}{\Delta_l} \le \prod_{l = 1}^{n - 2} \frac{\varphi(\rho_l)}{\rho_l} = \prod_{l = 1}^{n - 2} \prod_{\substack{p \le l \\ p \text{ is prime}}} (1 - p^{-1}) \\ &= \prod_{\substack{p \le n - 2 \\ p \text{ is prime}}} (1 - p^{-1})^{n - 1 - p}.
   \end{align*}
    On the other hand, if we take $a_1 = 1, a_2 = 2, \ldots, a_{n - 1} = n - 1$, but allow $a_n \rightarrow \infty$, then $\rho_l = \Delta_l$, for $l = 1, 2, \ldots, n - 2$. Thus,
    \[
        \limsup_{n \rightarrow \infty} \frac{\abs{\Cl(R)} \cdot 2^{n - 1}}{\varphi(\Delta)} = \prod_{\substack{p \le n - 2 \\ p \text{ is prime}}} (1 - p^{-1})^{n - 1 - p}. \qedhere
    \]
\end{proof}

%===============================================

\section{Examples of small degree}\label{sExamples}
We have obtained a closed formula for the order of $\Cl(R)$ in Theorem \ref{thmCl}. The goal 
of this section is to determine the structure of the group $\Cl(R)$ for $n=2$ and $n=3$. In fact, for $n=2$, we also compute the monoid  $\ICM(R)$. Extending these calculations to $n\geq 4$ seems quite complicated, although there are algorithms for doing so for any specific choice of $a_1, \dots, a_n$; see \cite{Marseglia}.  

\subsection{Quadratic case} 	
	Assume $n=2$. Denote $\chi(x)=(x-a)(x-b)$ with $a<b$, and, with abuse of notation, denote by $x$ the image of $x\in \Z[x]$ in $R=\Z[x]/(\chi(x))$. 
    %Hence $$R=\left\{\Z+\Z x\mid x^2=(a+b)x-ab\right\}. $$
	
	\begin{prop}\label{propICM2} The ideals 
		$$
		(x-b, 1), \quad (x-b, 2), \quad \dots \quad (x-b, \lfloor (b-a)/2\rfloor), \quad (x-b, b-a)  
		$$
		form a complete set of representatives of the equivalence classes of fractional ideals of $R$. The multiplication in $\ICM(R)$ is given by the 
		formula 
		$$
			(x-b, u)(x-b, v) \sim (x-b, uv/\gcd(u, v, a-b)). 
		$$
		Moreover, for any integer $w$ not divisible by $b-a$, we have 
		$$
		(x-b, w)=(x-b, -w) \sim (x-b, w\pm (b-a)). 
		$$
	\end{prop}
	\begin{proof}
	First, we make the isomorphisms used in the proof of Theorem \ref{thmLM} more explicit. 
	Consider $R$ as a free $\Z$-module, equipped with its natural $R$-module structure.  With respect to the $\Z$-basis $\{1, x\}$, the action of $x$ 
	on $\Z^2$ as on column vectors is given by the matrix $\begin{bmatrix}  & -ab\\ 1 & (a+b)\end{bmatrix}$. We extend this action to $V_1=\Q\otimes_\Z R\cong\Q^2$. 
	Next, let $M=\Z^2$ be the $R$-module on which $x$ acts by the matrix $\begin{bmatrix} a & u\\  & b\end{bmatrix}$; again we extend this action to $V_2=\Q\otimes_\Z M\cong\Q^2$. 
    
    The $K$-modules $V_1$ and $V_2$ are isomorphic, with the isomorphism $\theta\colon V_2\to V_1$ given by 
	a matrix $A\in \GL_2(\Q)$ such that  
    $$
	A \begin{bmatrix} a & u\\  & b\end{bmatrix} = \begin{bmatrix}  & -ab\\ 1 & (a+b)\end{bmatrix} A. 
	$$
    When $ab\neq 0$, one can take $A=\begin{bmatrix} ab & \\ -a & -u\end{bmatrix}$ if $u\neq 0$ or $A=\begin{bmatrix} b & a \\ -1 & -1\end{bmatrix}$ if $u=0$. When $a=0$, one can take $A=\begin{bmatrix} -b & u\\ 1 & \end{bmatrix}$ if $u\neq 0$ or $A=\begin{bmatrix} -b &  \\ 1 & 1\end{bmatrix}$ if $u=0$. When $b=0$, one 
    can take $A=\begin{bmatrix}  & u-a\\ 1 & 1\end{bmatrix}$ if $u\neq a$ or 
    $A=\begin{bmatrix}  & a\\ 1 & \end{bmatrix}$ if $u=a$. 
	
	Note that, by construction, $\theta(\Z^2)$ is an $R$-submodule of $K=V_1$. In fact, $\theta(\Z^2)$ lies in $R$, so it is an integral ideal $I_u\lhd R$. The generators 
	of $I_u$ are given by the images of the standard basis of $\Z^2$ under $\theta$. Assuming $ab\neq 0$ (the calculations are similar in the other cases), we get $I_u=(ab-ax, -ux)$ if $u\neq 0$, and $I_u=(a-x, b-x)$ if $u=0$. 
	Since $ab=-x^2+(a+b)x$, we have 
	$$
	(ab-ax, -ux) = (-x^2+(a+b)x-ax, ux) = (x(b-x), ux)\sim (x-b, u). 
	$$
	Also, note that $(a-x, b-x) = (x-b, b-a)$. From Theorem \ref{thmLM} and the proof of Lemma \ref{exampleICM2}, we conclude that the 
	ideals 
	$$
	(x-b, 1), \quad (x-b, 2), \quad \dots \quad (x-b, \lfloor (b-a)/2\rfloor), \quad (x-b, b-a)
	$$
	form a complete set of representatives of the equivalence classes of fractional ideals of $R$, and moreover  
	$(x-b, w)\sim (x-b, w\pm (b-a))$ for any integer $w$ not divisible by $(b-a)$. 
	
	Finally, we compute 
	\begin{align*}
	(x-b, u)(x-b, v) &= ((x-b)^2, (x-b)u, (x-b)v, uv) \\ 
	 & = (x^2-2xb+b^2, (x-b)\gcd(u, v), uv) \\ 
	  & = ((a+b)x-ab-2xb+b^2, (x-b)\gcd(u, v), uv) \\ 
	  & = ((a-b)x-(a-b)b, (x-b)\gcd(u, v), uv) \\ 
	  & = ((a-b)(x-b), (x-b)\gcd(u, v), uv) \\ 
	   & = ((x-b)\gcd(u, v, a-b), uv) \\ 
	   & \sim ((x-b), uv/\gcd(u, v, a-b)). \qedhere
	\end{align*}
	\end{proof}
	
\begin{thm}
	If $(b-a)\leq 2$, then $\Cl(R)$ is the trivial group. If $(b-a)>2$, then $\Cl(R)$ is isomorphic to the quotient of $(\Z/(b-a))^\times$ by the subgroup $\{\pm 1\}$.   
\end{thm}
\begin{proof} We use the notation of Proposition \ref{propICM2}. 
	Obviously $(x-b, 1)=R$ is the identity of $\Cl(R)$. Moreover, the ideal $(x-b, u)$ is invertible if and only if $u$ is coprime to $(b-a)$. 
	To see this, note that $u$ is coprime to $(b-a)$ if and only if there is $v$ such that $uv\equiv 1\Mod{b-a}$. For such $v$ we have 
	\begin{align*}(x-b, u)(x-b, v)&\sim (x-b, uv/\gcd(u, v, a-b))= (x-b, uv) \\ &\sim (x-b, 1)=R,
		\end{align*}
		where the equivalence relation is defined as in Proposition \ref{propICM2}. 
	There are exactly 
	$$
	\begin{cases}
		1, & \text{if }(b-a)=1,2; \\ 
		\frac{\varphi(b-a)}{2}  & \text{if }(b-a)>2
	\end{cases}
	$$
	numbers in the set $\{1,2, \dots, \lfloor (b-a)/2\rfloor, (b-a)\}$ that are coprime to $b-a$. Since this is the order of $\Cl(R)$, 
	we conclude that the ideals 
	$$
	\{ (x-b, u)\mid 1\leq u\leq (b-a)/2, \quad \gcd(u, b-a)=1\} 
	$$
	are a complete set of representatives of the ideal classes in $\Cl(R)$. 
	
	By the above discussion, the map $(\Z/(b-a))^\times\to \Cl(R)$, $u\mapsto (x-b, u)$, is a surjective homomorphism, whose kernel 
	contains $\{\pm 1\}$. Since the quotient of $(\Z/(b-a))^\times$ by $\{\pm 1\}$ has the same order as $\Cl(R)$, these groups are isomorphic. 
\end{proof}

\subsection{Cubic case} Assume $n=3$. Denote $\chi(x)=(x-a)(x-b)(x-c)$ with $a<b<c$. By the same abuse of notation as in the quadratic case we will denote by $x$ the image of $x\in \Z[x]$ in $R=\Z[x]/(\chi(x))$. In Remark \ref{unit_rem} it was discussed that $\Cl(R)$ is trivial when $c - a \le 3$. If $c-a\geq 4$, then 
\begin{equation} \label{formula_cl} 
    \abs{\Cl(R)} = \frac{\varphi(\Delta)}{4} \frac{\varphi(\Delta_1)}{\Delta_1},
\end{equation}
where $\Delta_1=\gcd(c-a, b-a, c-b)$. 
The following theorem gives a complete description for $\Cl(R)$. 
\begin{thm} \label{thm_clr}
    If $c - a \le 3$, then $\Cl(R)$ is trivial. If $c - a \ge 4$, then $\Cl(R)$ is isomorphic to the quotient of the group 
    \[
        G =\Big(\Z/(b - a)(c - a)\Big)^\times \times \Big(\Z/(c - b)(c - a)\Big)^\times
    \] 
    by the subgroup $N$ which can be represented as a direct product $HK$ of the subgroups $H$ and $K$, where
    \[
        H = \{(\overline{1}, \overline{1}), (\overline{1}, -\overline{1}), (-\overline{1}, \overline{1}), (-\overline{1}, -\overline{1})\}, 
    \]
    \[
        K = \Biggr \{ (\overline{u}, \overline{v}) \in G : b - a \mid u - 1, c - b \mid v - 1, c - a \mid \frac{u(v - 1)}{c - b} - \frac{u - 1}{b - a} \Biggr \}. 
    \]
\end{thm}
% The case $c - a \le 3$ was discussed in the previous section, so everywhere in this section we will assume that $c - a \ge 4$.
% \\
% \\
It turned out that for our proof of this theorem lower triangular matrices have some advantages over upper triangular matrices, so we shall work with those. 
Consider the following lower triangular integer matrix:
\begin{equation*}
    A = 
    \begin{bmatrix}
        a & & \\
        u & b &\\
        & v & c
    \end{bmatrix}.
\end{equation*}
\begin{prop}
    If $uv \ne 0$, then the $R$-ideal class of the fractional ideals of $R$ that corresponds to the $\Z$-conjugacy class of the matrix $A$ is the class of the ideal 
    $$I = (uv, v(x - a), (x - a)(x - b)).$$ %(Here $I$ is the ideal of $R$ generated by the listed three elements.)
\end{prop}
\begin{proof} The argument is a more complicated version of the argument in the proof of Proposition \ref{propICM2}. We omit the details. 
\end{proof}

Next we consider the map $\psi$ from the group $G$ to $\ICM(R)$ given by the rule
\begin{equation*}
    \psi\colon (\overline{u}, \overline{v}) \mapsto R\text{-ideal class of }(uv, v(x - a), (x - a)(x - b)).
\end{equation*}
(We will write $\psi(\overline{u}, \overline{v})$ instead of $\psi((\overline{u}, \overline{v}))$ to simplify the notation. Also, we will assume that the $\psi(u, v)$ is the fractional ideal written above and not its $R$-ideal class.)
\\
\\
First, we need to make sure that this map is well-defined and does not depend on the choice of representatives from the congruence classes of $u$ and $v$. Thus, we need to prove the following:
\begin{prop}
    For pairs of integers $(u, v)$ and $(u', v')$, with $uv, u'v' \neq 0$ and $u \equiv u' \pmod{(b - a)(c - a)}$ and $v \equiv v' \pmod{(c - b)(c - a)}$, the matrices
    \begin{equation*}
        A = 
        \begin{bmatrix}
            a & & \\
            u & b &\\
            & v & c
        \end{bmatrix}
        \quad 
        \text{and}
        \quad
        A' = 
        \begin{bmatrix}
            a & & \\
            u' & b &\\
            & v' & c
        \end{bmatrix}
    \end{equation*}
    are conjugate over $\Z$. 
\end{prop}
\begin{proof} This is based on straightforward calculations, which we omit. 
\begin{comment}
    It suffices to show that there are integers $r, s, t$ such that 
    \begin{equation*}
        \begin{bmatrix}
            1 & &\\
            r & 1 &\\
            t & s & 1
        \end{bmatrix}
        \begin{bmatrix}
            a & & \\
            u & b &\\
            & v & c
        \end{bmatrix}
        =
        \begin{bmatrix}
            a & & \\
            u' & b &\\
            & v' & c
        \end{bmatrix}
        \begin{bmatrix}
            1 & &\\
            r & 1 &\\
            t & s & 1
        \end{bmatrix}. 
    \end{equation*}
    The later is equivalent to 
    \begin{equation*}
        \begin{bmatrix}
            a & &\\
            ar + u & b &\\
            at + us & bs + v & c
        \end{bmatrix}
        =
        \begin{bmatrix}
            a & &\\
            br + u' & b &\\
            ct + v'r & cs + v' & c
        \end{bmatrix}. 
    \end{equation*}
    This translates to the following system of linear equations 
    \begin{equation*}
        \begin{cases}
            (b - a)r = u - u',\\
            (c - b)s = v - v',\\
            (c - a)t = us - v'r.
        \end{cases}
    \end{equation*}
   This system has a solution if and only if 
    \begin{equation*}
        b - a \mid u - u', \qquad c - b \mid v - v', \qquad c - a \mid uy - v'x,
    \end{equation*}
    which is true when $u \equiv u' \pmod{(b - a)(c - a)}$ and $v \equiv v' \pmod{(c - b)(c - a)}$.
    \end{comment}
\end{proof}
Now, we can stop using the bar notation for the elements of $\Z/(b - a)(c - b)$ and $\Z/(c - b)(c - a)$. Next, we claim that $\psi$ is a homomorphism of commutative monoids (note that we also have $\psi(1, 1) = R$). 
\begin{prop}
    For any elements $(u, v)$ and $(u', v')$ of $G$, we have 
    $$\psi(u, v)\psi(u', v') = \psi(uu', vv').$$ In other words, $\psi \colon G \rightarrow \ICM(R)$ is a homomorphism of commutative monoids. 
\end{prop}
\begin{proof}
    Observe that
    \begin{align*}
        \psi(u, v)\psi(u', v') = &(uv, v(x - a), (x - a)(x - b))(u'v', v'(x - a), (x - a)(x - b))\\
        = &\Big(uu'vv', vv'u'(x - a), u'v'(x - a)(x - b),\\
        &\ vv'u(x - a), vv'(x - a)^2, v'(c - a)(x - a)(x - b),\\
        &\ uv(x - a)(x - b), v(c - a)(x - a)(x  - b), (c - a)(c - b)(x - a)(x - b) \Big)\\
        = &\Big(uu'vv', vv'\gcd(u, u')(x - a), vv'(x - a)^2,\\
        &\ \gcd(uv, u'v', v(c - a), v'(c - a), (c - a)(c - b))(x - a)(x - b) \Big). 
    \end{align*}
    Since $\gcd(v, c - b) = 1$, the last generator can be rewritten as $\gcd(uv, u'v', c - a)(x - a)(x - b)$. Since $\gcd(u, c - a) = \gcd(v, c - a) = 1$, $\gcd(uv, c - a) = 1$ and hence the mentioned generator simplifies to $(x - a)(x - b)$. Thus, 
    \begin{align*}
        \psi(u, v)\psi(u', v') &= \Big(uu'vv', vv'\gcd(u, u')(x - a), vv'(x - a)^2, (x - a)(x - b) \Big) \\
        &= \Big(uu'vv', vv'\gcd(u, u')(x - a), vv'(b - a)(x - a), (x - a)(x - b) \Big) \\
        &= \Big(uu'vv', vv'\gcd(u, u', b - a)(x - a), (x - a)(x - b) \Big) \\
        &= \Big(uu'vv', vv'(x - a), (x - a)(x - b) \Big) \\
        &= \psi(uu', vv'). \qedhere
    \end{align*}
\end{proof}

It follows that $\psi \colon G \rightarrow \Cl(R)$ is a homomorphism of abelian groups. Next, we describe $\ker \psi$ and find its order.
\begin{prop} \label{ker_ord_prop}
    $K$ is a subgroup of $G$ and $\ker \psi$ is a direct product of $H$ and $K$. Furthermore, $\abs{H}= 4$ and $\abs{K}= \varphi(c - a) \frac{\Delta_1}{\varphi(\Delta_1)}$, so $\abs{\ker \psi} = 4\varphi(c - a) \frac{\Delta_1}{\varphi(\Delta_1)}$.
\end{prop}
\begin{proof}
    First of all, we note that $(1, 1), (-1, 1), (1, -1), (-1, -1) \in \ker \psi$. Since $c - a \ge 4$ all these four elements are distinct. Hence $\abs{H} = 4$ and $H \subseteq \ker \psi$. Next, suppose $(u, v) \in \ker \psi$. Then there must exist integers $r, s, t$ and $\varepsilon_1, \varepsilon_2, \varepsilon_3 \in \{\pm 1\}$, such that
    \begin{equation*}
        \begin{bmatrix}
            \varepsilon_1 & &\\
            r & \varepsilon_2 & &\\
            t & s & \varepsilon_3
        \end{bmatrix}
        \begin{bmatrix}
            a & &\\
            u & b & \\
            & v & c\\
        \end{bmatrix}
        =
        \begin{bmatrix}
            a & &\\
            1 & b & \\
            & 1 & c\\
        \end{bmatrix}
        \begin{bmatrix}
            \varepsilon_1 & &\\
            r & \varepsilon_2 & &\\
            t & s & \varepsilon_3
        \end{bmatrix}. 
    \end{equation*}
    This matrix equation translates to
    \begin{equation*}
        \begin{cases}
            u\varepsilon_2 - \varepsilon_1 = (b - a)r,\\
            v\varepsilon_3 - \varepsilon_2 = (c - b)s,\\
            us - r = (c - a)t.
        \end{cases}
    \end{equation*}
    In particular, it follows that $u \equiv \pm 1 \pmod{b - a}$ and $v \equiv \pm 1 \pmod{c - b}$. Multiplying $(u, v)$ by a suitable element of $H$, we can achieve $u \equiv 1 \pmod{b - a}$ and $v \equiv 1 \pmod{c - b}$. This implies $\varepsilon_1 = \varepsilon_2 = \varepsilon_3$. Without loss of generality we may assume that $\varepsilon_1 = \varepsilon_2 = \varepsilon_3 = 1$. Thus, for the element $(u, v) \in G$, we have
    \begin{equation}\label{eqDivs}
        b - a \mid u - 1, \qquad c - b \mid v - 1, \qquad c - a \mid \frac{u(v -  1)}{c - b} - \frac{u - 1}{b - a},
    \end{equation}
    which means $(u, v) \in K$. 
    
    Next we will show that $K$ is a subgroup of $G$. It is obvious that $(1, 1) \in K$. Since $G$ is a finite group, it suffices to show that $K$ is closed under multiplication. Suppose $(u, v), (u', v') \in K$. Then 
    \begin{align*}
        (b - a)u(v - 1) &\equiv (c - b)(u - 1) \pmod{\Delta}, \\
        (b - a)u'(v' - 1) &\equiv (c - b)(u' - 1) \pmod{\Delta}.
    \end{align*}
    Note that modulo $\Delta$ we have
    \begin{align*}
        (b - a)uu'(vv' - 1) &= u'v'(b - a)u(v - 1) + u(b - a)u'(v' - 1) \\
        &= u'v'(c - b)(u - 1) + u(c - b)(u' - 1)\\
        &= (c - b)(u'v' - 1)(u - 1) + (c - b)(uu' - 1). 
    \end{align*}
    So it suffices to show that $c - a \mid u'v' - 1$. Note that
    \[
        \Delta \mid (b - a)u'(v' - 1) - (c - b)(u' - 1) = (b - a)(u'v' - 1) - (c - a)(u' - 1). 
    \]
    Since $b - a \mid u' - 1$, we see that $c - a \mid u'v' - 1$ as desired. So, $K \leq G$. 

    The discussion at the beginning of the proof shows that $\ker \psi = HK$ and it is not difficult to see that $H \cap K = \{1\}$. So $\ker \psi$ is a direct product of the subgroups $H$ and $K$. 
   
    Now suppose we have chosen some $u \in \Big(\Z/(b - a)(c - a) \Big)^\times$ such that $u \equiv 1 \pmod{b - a}$. 
    The third divisibility in \eqref{eqDivs} implies that $v$ is determined uniquely. We need to make sure that this 
    integer $v$ satisfies $\gcd(v , c - a) = 1$. Suppose there is some prime number $p$ such that $p \mid v$ and $p\mid c - a$. Multiplying the third divisibility 
    relation by $(b - a)(c - b)$, we get $c - a \mid u(v - 1)(b - a) - (u - 1)(c - b)$. In particular, it follows that $c - a \mid (uv - 1)(c - b)$. Then $\frac{c - a}{\Delta_1} \mid uv - 1$. Since $v \equiv 1 \pmod{c - b}$, we must have $p \mid \frac{c - a}{\Delta_1}$, which is a contradiction. 
   
    It remains to find the number of elements of $\Big(\Z/(b - a)(c - a) \Big)^\times$ such that $u \equiv 1 \pmod{b - a}$. To do this, consider the homomorphism 
    $$\rho \colon \Big(\Z/(b - a)(c - a) \Big)^\times \rightarrow \Big( \Z/(b - a) \Big)^\times$$ given by the rule $\rho\colon u \mod{(c - a)(b - a)} \mapsto u \mod{(b - a)}$. The number we are looking for is equal to $\abs{\ker \rho}$. We will show that $\rho$ is surjective. We need to show that 
    for a fixed $w \in \Big( \Z/(b - a) \Big)^\times$ there is an integer $k$ such that $\gcd(w + k(b - a), c - a) = 1$. Note that if a prime number $p$ divides $\gcd(w + k(b - a), c - a)$, then we must have $p \nmid b - a$. Let $p_1, p_2, \dots, p_l$ denote those prime divisors of $c - a$ 
     that do not divide $b - a$. By the Chinese Remainder Theorem, we can find $k_0 \in \Z$, such that $w + k_0(b - a) \equiv 1 \pmod{p_j}$, for $j = 1, 2, \dots, l$. So $\rho$ is surjective and hence $$\abs{\ker \rho} = \frac{\varphi((b - a)(c - a))}{\varphi(b - a)} = \varphi(c - a)\frac{\Delta_1}{\varphi(\Delta_1)}.$$ (Here we use the formula $\varphi(mn) = \varphi(m) \varphi(n) \frac{\gcd(m, n)}{\varphi(\gcd(m, n))}$.) 
   We conclude that $\abs{\ker \psi} = 4\abs{\ker \rho} = 4\varphi(c - a)\frac{\Delta_1}{\varphi(\Delta_1)}$.

Using the same formula $\varphi(mn) = \varphi(m) \varphi(n) \frac{\gcd(m, n)}{\varphi(\gcd(m, n))}$, it can be shown that $\varphi((b - a)(c - a))\varphi((c - b)(c - a)) = \varphi(\Delta)\varphi(c - a)$, so $\abs{G} = \varphi(\Delta)\varphi(c - a)$. Now using the First Isomorphism Theorem, we see that $\abs{\im \psi} = \frac{\varphi(\Delta)\varphi(\Delta_1)}{4 \Delta_1}$. Thus, by \eqref{formula_cl}, $\psi$ is surjective. Therefore $\Cl(R) \cong G/\ker \psi = G / HK$, which concludes the proof of Theorem \ref{thm_clr}.
\end{proof}

For $\Delta_1 = 1$, there is a simpler description of $\Cl(R)$. 
\begin{thm}
    If $c - a \ge 4$ and $\Delta_1 = 1$, then $\Cl(R)$ is isomorphic to the quotient of the group
    \[
        G' = \Big ( \Z/(b - a) \Big )^\times \times \Big ( \Z/(c - b) \Big )^\times \times \Big ( \Z/(c - a) \Big )^\times
    \]
    by the subgroup
    \[
        H' = \Big\{ (\overline{1}, \overline{1}, \overline{1}), (\overline{1}, -\overline{1}, -\overline{1}), (-\overline{1}, \overline{1}, -\overline{1}), (-\overline{1}, -\overline{1}, \overline{1}) \Big \}.
    \]
\end{thm}
\begin{proof}
    In this case
    \[
        K = \Big \{\ (\overline{u}, \overline{v})\ :\  b - a \mid u - 1, c - b \mid v - 1, c - a \mid uv - 1 \Big \}.
    \]
    Consider the map $\pi: G \rightarrow G'$ defined by the rule 
    $$(u \bmod{(b - a)(c - a)}, v \bmod{(c - b)(c - a)}) \longmapsto (u \bmod{b - a}, v \bmod{c - b}, uv \bmod{c - a}).$$ It is not difficult to check that this map is a group homomorphism and $\ker \pi = K$. By Proposition \ref{ker_ord_prop}, $|K| = \varphi(c - a)$. Thus, 
    \begin{align*}
        |G/K| = \frac{|G|}{|K|} &= \frac{\varphi(b - a)\varphi(c - a) \varphi(c - b)(c - a)}{\varphi(c - a)} \\ &= \varphi(b - a)\varphi(c - b)\varphi(c - a) = |G'|.
    \end{align*}
    Therefore, $\pi$ is surjective and $G/K \cong G'$. Now note that $\pi(KH) = H'$. By the Third Isomorphism Theorem, $\Cl(R) \cong G/KH \cong G'/H'$. 
\end{proof}
%\begin{cor}
   % If $\Delta_1 = 1$, then $\Cl(R)$ is isomorphic to a quotient of the group $\Big( \Z/(\Delta) \Big)^\times$.
%\end{cor}

%---------------------------------------------------------
%\renewcommand{\bibliofont}{\normalsize}
\bibliographystyle{amsalpha}
\bibliography{Bibliography.bib}

@book {DF,
	AUTHOR = {Dummit, David S. and Foote, Richard M.},
	TITLE = {Abstract algebra},
	EDITION = {Third},
	PUBLISHER = {John Wiley \& Sons, Inc., Hoboken, NJ},
	YEAR = {2004},
	PAGES = {xii+932},
}

@book {Marcus,
	AUTHOR = {Marcus, Daniel A.},
	TITLE = {Number fields},
	SERIES = {Universitext},
	PUBLISHER = {Springer, Cham},
	YEAR = {2018},
	PAGES = {xviii+203},
}

@article {KL,
	AUTHOR = {Kopp, Gene S. and Lagarias, Jeffrey C.},
	TITLE = {Ray class groups and ray class fields for orders of number
	fields},
	JOURNAL = {Essent. Number Theory},
	FJOURNAL = {Essential Number Theory},
	VOLUME = {4},
	YEAR = {2025},
	NUMBER = {1},
	PAGES = {1--65},
}

@book {Cohen,
	AUTHOR = {Cohen, Henri},
	TITLE = {A course in computational algebraic number theory},
	SERIES = {Graduate Texts in Mathematics},
	VOLUME = {138},
	PUBLISHER = {Springer-Verlag, Berlin},
	YEAR = {1993},
	PAGES = {xii+534},
}

@article {BHJ,
	AUTHOR = {Bley, Werner and Hofmann, Tommy and Johnston, Henri},
	TITLE = {Computation of lattice isomorphisms and the integral matrix
	similarity problem},
	JOURNAL = {Forum Math. Sigma},
	FJOURNAL = {Forum of Mathematics. Sigma},
	VOLUME = {10},
	YEAR = {2022},
	PAGES = {Paper No. e87, 36},
}

@article {JP,
	AUTHOR = {Jordan, Bruce W. and Poonen, Bjorn},
	TITLE = {The analytic class number formula for 1-dimensional affine
	schemes},
	JOURNAL = {Bull. Lond. Math. Soc.},
	FJOURNAL = {Bulletin of the London Mathematical Society},
	VOLUME = {52},
	YEAR = {2020},
	NUMBER = {5},
	PAGES = {793--806},
}

@article {LM,
	AUTHOR = {Latimer, Claiborne G. and MacDuffee, C. C.},
	TITLE = {A correspondence between classes of ideals and classes of
	matrices},
	JOURNAL = {Ann. of Math. (2)},
	FJOURNAL = {Annals of Mathematics. Second Series},
	VOLUME = {34},
	YEAR = {1933},
	NUMBER = {2},
	PAGES = {313--316},
}

@article {Conrad,
	AUTHOR = {Conrad, Keith},
	TITLE = {Ideal classes and matrix conjugation over {$\Z$}},
	note={available at \texttt{https://kconrad.math.uconn.edu/blurbs}}
}

@article {Marseglia,
	AUTHOR = {Marseglia, Stefano},
	TITLE = {Computing the ideal class monoid of an order},
	JOURNAL = {J. Lond. Math. Soc. (2)},
	FJOURNAL = {Journal of the London Mathematical Society. Second Series},
	VOLUME = {101},
	YEAR = {2020},
	NUMBER = {3},
	PAGES = {984--1007},
}

@book {Neukirch,
	AUTHOR = {Neukirch, J\"{u}rgen},
	TITLE = {Algebraic number theory},
	SERIES = {Grundlehren der mathematischen Wissenschaften},
	VOLUME = {322},
	PUBLISHER = {Springer-Verlag, Berlin},
	YEAR = {1999},
	PAGES = {xviii+571},
}

@book {Cox,
	AUTHOR = {Cox, David A.},
	TITLE = {Primes of the form {$x^2+ny^2$}---{F}ermat, class field
	theory, and complex multiplication},
	PUBLISHER = {AMS Chelsea Publishing, Providence, RI},
	YEAR = {2022},
	PAGES = {xv+533},
}

@article {EHO,
    AUTHOR = {Eick, Bettina and Hofmann, Tommy and O'Brien, E. A.},
     TITLE = {The conjugacy problem in {$\GL(n, \Z)$}},
   JOURNAL = {J. Lond. Math. Soc. (2)},
  FJOURNAL = {Journal of the London Mathematical Society. Second Series},
    VOLUME = {100},
      YEAR = {2019},
    NUMBER = {3},
     PAGES = {731--756},
      ISSN = {0024-6107},
}

\end{document}